\documentclass[11pt,reqno]{amsart}
\usepackage{amsmath, amsthm,amssymb,bbm,enumerate,mathrsfs,mathtools}
\usepackage[margin = 2.5cm]{geometry}
\usepackage{tikz,xcolor,verbatim}
\usetikzlibrary{calc,shapes}
\usepackage{hyperref}
\usepackage{makecell}
\usepackage{framed}
\usepackage{setspace}
\usepackage{cleveref}
\usepackage{caption}

\usepackage{enumitem}
\setlist[enumerate]{itemsep=3pt,topsep=3pt}
\setlist[enumerate,1]{label=\rm{(\roman*)}}

\usepackage[square,sort,comma,numbers]{natbib} 
\tikzset{
  bigblack/.style={circle, draw=black!100,fill=black!100,thick, inner sep=1.5pt, minimum size=6mm},
  medblack/.style={circle, draw=black!100,fill=black!100,thick, inner sep=1.5pt, minimum size=4mm},
  blackcirc/.style={circle, draw=black!100,thick, inner sep=1.5pt, minimum size=6mm},  
}

\makeatletter
\renewcommand\subsection{\@startsection{subsection}{2}%
  \z@{-.5\linespacing\@plus-.7\linespacing}{.5\linespacing}%
  {\normalfont\bfseries}}
\makeatother

\allowdisplaybreaks

\date{}
\newtheoremstyle{case}{}{}{\normalfont}{}{\itshape}{:}{ }{}

\theoremstyle{plain}

\newtheorem*{thm*}{Theorem}
\newtheorem{thm}{Theorem}
\Crefname{thm}{Theorem}{Theorems}
\numberwithin{thm}{section}

\newtheorem*{lem*}{Lemma}
\newtheorem{lem}[thm]{Lemma}
\Crefname{lem}{Lemma}{Lemmas}

\newtheorem*{claim*}{Claim}
\newtheorem{claim}[thm]{Claim}
\crefname{claim}{Claim}{Claims}
\Crefname{claim}{Claim}{Claims}

\newtheorem{prop}[thm]{Proposition}
\Crefname{prop}{Proposition}{Propositions}

\newtheorem{cor}[thm]{Corollary}
\crefname{cor}{Corollary}{Corollaries}

\newtheorem{conj}[thm]{Conjecture}
\crefname{conj}{Conjecture}{Conjectures}

\Crefname{qn}{Question}{Questions}

\Crefname{obs}{Observation}{Observations}

\Crefname{ex}{Example}{Examples}

\theoremstyle{definition}

\Crefname{prob}{Problem}{Problems}

\Crefname{defn}{Definition}{Definitions}

\theoremstyle{remark}
\newtheorem{rem}[thm]{Remark}
\Crefname{rem}{Remark}{Remarks}

\usepackage{xpatch}
\xpatchcmd{\proof}{\itshape}{\normalfont\proofnamefont}{}{}
\newcommand{\proofnamefont}{}
\renewcommand{\proofnamefont}{\bfseries}

\numberwithin{equation}{section}

\numberwithin{equation}{section}

\newcommand{\A}{\mathcal{A}}
\newcommand{\B}{\mathcal{B}}
\newcommand{\N}{\mathbb{N}}
\newcommand{\R}{\mathbb{R}}

\newcommand{\F}{\mathcal{F}}
\renewcommand{\S}{\mathcal{S}}

\newcommand{\eps}{\varepsilon}

\newcommand{\col}[2]{\mathcal{C}(#1,#2)} %colex
\newcommand{\lex}[3]{\mathcal{L}(#1,#2,#3)}
\newcommand{\cl}[2]{[#1]^{(#2)}}
\DeclareMathOperator{\lexx}{lex}

\begin{document}

\title[Lagrangians of Hypergraphs II: When colex is best]{Lagrangians of Hypergraphs II: \\ When colex is best}
\author{Vytautas Gruslys \and Shoham Letzter \and Natasha Morrison}

\address{Department of Pure Mathematics and Mathematical Statistics,        University of Cambridge, Wilberforce Road, CB3 0WB Cambridge,        UK;}\email{v.gruslys|morrison@dpmms.cam.ac.uk}

\address{ETH Institute for Theoretical Studies, ETH, 8092 Zurich}\email{shoham.letzter@eth-its.ethz.ch}

\thanks{The second author is supported by Dr.~Max
        R\"ossler, the Walter Haefner Foundation and by the ETH Zurich Foundation. The third author is supported by a research fellowship from Sidney Sussex College, Cambridge.}

    \begin{abstract}
		A well-known conjecture of Frankl and F\"{u}redi from 1989 states that an initial segment of the colexicographic order has the largest Lagrangian of any $r$-uniform hypergraph with $m$ hyperedges. We show that this is true when $r=3$. We also give a new proof of a related conjecture of Nikiforov for large $t$ and a counterexample to an old conjecture of Ahlswede and Katona.
    \end{abstract}

\maketitle

\section{Introduction}

	The notion of the Lagrangian of a graph was originally introduced in 1965 by Motzkin and Straus~\cite{MotStr} to provide a beautiful new proof of Tur\'{a}n's theorem. This concept was later generalised to uniform hypergraphs, where the study of Lagrangians has played an important role in the advancement of our understanding of hypergraph Tur\'{a}n problems. Notably, hypergraph Lagrangians were used by Frankl and R\"{o}dl~\cite{FraRod} to disprove a conjecture of Erd\H{o}s~\cite{Erd} on jumps of hypergraph Tur\'{a}n densities. See, for example, \cite{HefKee,BenNorYep,BraIrwJia} and the excellent survey of Keevash~\cite{Kee} for further applications. 
		
	In many of these results, the Tur\'{a}n problem can be converted into the problem of determining (or finding good bounds for) the Lagrangian of a particular hypergraph. In this paper we are interested in determining the maximum value, over all $r$-uniform hypergraphs (which we call \emph{$r$-graphs}) with a given number of edges, of the Lagrangian, which we will now define. 
	
	For a set $X$, let $X^{(r)}:= \{X \subseteq \N: |X| = r\}$. Let $G \subseteq \N^{(r)}$ be a finite $r$-graph. We define the \emph{Lagrangian} of $G$ to be
	\begin{equation}\label{def:lag}
		\lambda(G) 
		:= \max \left\{\sum_{e \in E(G)}\prod_{x \in e}w(x): w(x) \in \R^{\ge 0} \text{ for all } x \in \N, \text{ and } \sum_{x \in \N}w(x) = 1\right\}.
	\end{equation}
	
	Observe that this maximum always exists by compactness. When $r=2$ (i.e.\ $G$ is a graph), it is not difficult to see that $\lambda(G)$ is obtained by spreading the weight equally over the vertices of a clique of maximum size in $G$ (see, e.g.\ Motzkin and Straus \cite{MotStr}). But for $r \ge 3$, it is generally hard to determine the Lagrangian of a given $r$-graph.

	A conjecture of Frankl and F\"{u}redi~\cite{FraFur} from 1989 states that for any $r$-graph $G$ with $|G| = m$, 
	\begin{equation}\label{conj:frankl-furedi}
		\lambda(G) \le \lambda(\col{m}{r}),
	\end{equation}
	where $\col{m}{r}$ is the family containing the first $m$ sets in the colexicographic order, in which $A < B$ if and only if $\sum_{i \in A} 2^i < \sum_{i \in B} 2^i$.
	
	The result of Motzkin and Straus~\cite{MotStr} confirms the truth of this conjecture when $r=2$. In \cite{Us1}, we disproved this conjecture by finding an infinite family of counterexamples for all $r \ge 4$. In this paper we resolve the remaining case, that is, when $r=3$. Our main theorem is the following.
	
	\begin{thm}\label{thm:r3}
		Let $m$ be sufficiently large. Then for every $3$-graph $G$ with $m$ edges  
		\[	
			\lambda(G) \le \lambda(\col{m}{3}).
		\]
	\end{thm}

	Before discussing the proof of \Cref{thm:r3}, let us briefly summarise the body of previous work that built up to it.
	When discussing results in support of Frankl and F\"uredi's conjecture \eqref{conj:frankl-furedi}, it is helpful to simultaneously consider the ranges $\binom{t}{r} \le m < \binom{t+1}{r}$, for $t \in \N$ (where $t$ is sometimes taken to be sufficiently large). This is a sensible division, as when $m = \binom{t}{r}$, the colex graph $\col{m}{r}$ is the clique $\cl{t}{r}$. 
	
	The first progress on Frankl and F\"uredi's conjecture was made in 2002 by Talbot, who proved the conjecture in the case $r = 3$ and and $\binom{t}{3} - 2 \le m \le \binom{t+1}{3} - (2t-1)$ for any $t \in \N$. More progress was made by Tang, Peng, Zhang and Zhao~\cite{TanPen,TanPen2}, who extended this range to $\binom{t}{3}-4 \le m \le \binom{t+1}{3} - (\frac{3t}{2} - 1)$.
	The range was further increased by Tyomkyn~\cite{Tyo} and Lei, Lu and Peng~\cite{LLP} who respectively increased it to all but the largest $t + O(t^{3/4})$ and $t + O(t^{2/3})$ values of $m$ in the interval $[\binom{t}{3}, \binom{t+1}{3}]$. In \cite{Us1}, we proved the following.

	\begin{thm}[\cite{Us1}]\label{thm:main1}		
		Let $r \ge 3$, let $t \in \N$ be large, and let $m$ satisfy $\binom{t-1}{r} \le m \le \binom{t}{r} - \binom{t-2}{r-2}$. Let $G$ be an $r$-graph with $m$ edges. Then $\lambda(G) \le \lambda(\col{m}{r})$.	
	\end{thm}
 
	This theorem contains all cases of Frankl and F\"uredi's conjecture confirmed by earlier work (except for small values of $t$, and the case $r = 3$ and $m = \binom{t}{3} - a$, where $a \in [4]$).
	The most noteworthy such work, apart from the above progress on the $r = 3$ case, is a paper by Tyomkyn~\cite{Tyo} who proved the conjecture for $r \ge 4$ and `almost all' values of $m$. Other relevant works on the case $r \ge 4$ include Nikiforov~\cite{Nik}, Lu~\cite{Lu} and Lei and Lu~\cite{LeiLu}. See \cite{Us1} for a more detailed overview of their work. 
	
	Our main theorem applies for all $r \ge 3$ and, together with \Cref{thm:main1}, significantly extends the range of values of $m$ for which \eqref{conj:frankl-furedi} is known to hold. 
	
	\begin{thm}\label{thm:range}
		Let $r \ge 3$, let $t \in \N$ be sufficiently large, let $r+2 \le a \le t - (r-1)$ and set $m := \binom{t}{r} - a$. Suppose that $G$ maximises the Lagrangian among $r$-graphs with at most $m$ edges. Then $G$ is isomorphic to the colex graph $\col{m}{r}$. 
	\end{thm}	

	Observe that \Cref{thm:r3} follows immediately from this and \Cref{thm:main1}, using the aforementioned results of Talbot~\cite{Tal} and Tang, Peng, Zhang, and Zhao~\cite{TanPen} who proved the statement when $m = \binom{t}{r} - a$, for $a \in \{1,2,3,4\}$.
	
	Thus Frankl and F\"uredi's conjecture is resolved; it holds for $r \in \{2,3\}$ and is false for all $r \ge 4$. However, Nikiforov~\cite{Nik} noted that \eqref{conj:frankl-furedi} does not provide an explicit expression for $\lambda(\col{m}{r})$; in light of this he made the following conjecture and proved it for $3 \le r \le 5$ and sufficiently large $m$, using analytic arguments.
	
	\begin{conj}[Nikiforov~\cite{Nik}]
	\label{conj:nik}
		Let $r \ge 3$. If $m = \binom{x}{r}$, for some real $x$ which satisfies $x \ge r-1$, then
		$\lambda(m,r) \le m  x^{-r}$, with equality if and only if $x \in \N$. 
	\end{conj}
	
	This conjecture was recently proven for all $r$ and $m$ by Lu~\cite{Lu} using a different analytic approach to that of Nikiforov. We have an independent\footnote{We uploaded to arXiv a preliminary version of this paper (see \href{https://arxiv.org/abs/1807.00793v2}{\textcolor{blue}{arXiv1807.00793v2}}), which also includes the content of our other paper \cite{Us1}, a few months after Lu's paper appeared on arXiv.} proof of this conjecture for $r \ge 3$ and sufficiently large $t$, which follows directly from the methods we use to prove \Cref{thm:range}. We include this result in an appendix; we believe that, as we use very different techniques to Lu, our proof is of independent interest. 	

	\subsection{Maximising the sum of degrees squared}\label{subsec:sqs}
	
		The proof of \Cref{thm:range} relies on exploiting the relationship between the problem of maximising the Lagrangian and the problem of maximising the sum of degrees squared. We remark that this relationship is also explored in \cite{Us1}, where we used it to find an infinite family of counterexamples to \eqref{conj:frankl-furedi} for $r \ge 4$. 

		We now define the \emph{sum of degrees squared} of a hypergraph $H$, denoted $P_2(H)$, by setting 
		\[
			P_2(H) := \sum_{x \in V(H)} d(x)^2,
		\] 
		where $d(x)$ is the degree of $x \in H$ (i.e.\ the number of edges in $H$ that contain $x$).
		We shall be interested in the following parameters, that maximise $P_2(\cdot)$ under some conditions.
		\begin{align*}
			& P_2(r,m):= \max \{P_2(H): H \subseteq \N^{(r)}, |H| = m\},  \text{ and }\\
			& P_2(r,m,t):= \max \{P_2(H): H \subseteq [t]^{(r)}, |H| = m\}.
		\end{align*}		

		We will show (see \Cref{cor:complement-max-degrees-squared} below) that,  under certain conditions, if $G$ is an $r$-graph that maximises the Lagrangians among $r$-graphs with the same size, then the complement of $G$ must be \emph{close} to maximising the sum of degrees squared. This along with \Cref{prop:i-1-intersecting-families} (stated below and proved in \Cref{sec:sum-degs-squared}) will give the proof of \Cref{thm:range} when $\binom{t}{r} - m$ is `small'.
		
		\begin{prop} \label{prop:i-1-intersecting-families}
			Let $H$ be an $r$-graph that maximises $P_2(\cdot)$ among $r$-graphs with the same number of edges. Then either $H$ is isomorphic to a subgraph of the clique $\cl{r+1}{r}$, or
			there is a set of $r-1$ vertices which is contained in every edge of $H$.
		\end{prop}    
		
	   \Cref{prop:i-1-intersecting-families} characterises the $r$-graphs $H$ that satisfy $P_2(H) = P_2(r,|H|)$. We could not find such a result in the literature, but as it is a natural problem, it is not implausible that such a solution exists. Related results of Friedgut and Kahn~\cite{FriKah} consider the problem of maximising the number of labelled copies of a hypergraph over all hypergraphs with a fixed number of edges. This generalises earlier work of Alon~\cite{Alo} who considered the problem in graphs (note that the sum of degrees squared in a graph $G$ is essentially the number of labelled copies of $K_{1,2}$). 
	   
	   In contrast, the problem of characterising $r$-graphs $H$ on $t$ vertices for which $P_2(H) = P_2(r,|H|,t)$ has drawn considerable attention. For $r = 2$, Ahlswede and Katona \cite{AhlKat} and, independently (but significantly later), Olpp \cite{Olp} showed that for every $m$ and $t$ either the colex graph $\col{m}{2}$ or the lex graph $\lex{m}{t}{2}$ graph are maximisers of $P_2(H)$, among $t$-vertex graphs with $m$ edges. (The lex graph $\lex{m}{t}{r}$ is defined as follows: given sets $A, B \in [t]^{(r)}$, $A <_{\lexx} B$ if and only if $\min\{A, B\} \in A$. The graph $\lex{m}{t}{r}$ is the initial segment according to $<_{\lexx}$ of $[t]^{(r)}$ of size $m$.) 
			
		Characterising the maximisers precisely is a surprisingly delicate task (see, e.g.\ \cite{AhlKat,PelPetSte,AbrFerNeuWat}). For $r \ge 3$, even the task of calculating $P_2(r,m,t)$ seems out of reach. Ahlswede and Katona~\cite{AhlKat} conjectured that when $r=3$, the sum of degrees squared is maximised by an initial segment of lex or colex.
		We found a counterexample to their conjecture by computer search. Take 
		\[
			H:= \{123,\, 124,\, 125,\, 126,\, 127,\, 134,\, 135,\, 136,\, 145,\, 146,\, 156\} \subseteq [7]^{3}.
		\]
		One can check that $P_2(H) = 211$, while any lex $3$-graph or a complement of a lex $3$-graph (this includes all colex graphs) with $11$ edges has sum of degrees squared at most $209$. So it is not the case that for every $r,m,t$ either the corresponding colex or lex graph are maximisers of $P_2(r,m,t)$. Nevertheless, some upper bounds on $P_2(r,m,t)$ have been proved \cite{Cae,Nik07,Bey}. 

		In the next section we gather together notation, definitions and some preliminary bounds that will be used throughout the paper. The proof of \Cref{thm:range} is divided into two regimes; where $\binom{t}{r} - m$ is `large' (i.e.\ at least $t^{0.01}$), considered in Section~\ref{sec:a-small} and where it is `small', considered in~\Cref{sec:a-med}. We remind the reader that our proof of \Cref{prop:i-1-intersecting-families} is included in \Cref{sec:sum-degs-squared} and the proof of \Cref{conj:nik} is contained in \Cref{app:nik}.

\section{Preliminaries}\label{sec:prelims}

	Here we collect together notation and definitions that we shall use throughout the paper, as well as preliminary results that will be used in the proof of \Cref{thm:range}.
	In \Cref{subsec:defs} we provide notation and definitions, in \Cref{subsec:prelims-old} we state some relevant results from previous papers, and in \Cref{subsec:prelims-new} we provide additional preliminary results, along with their proofs. We note that all of the content of \Cref{subsec:defs,subsec:prelims-old} appears in \cite{Us1} together with the proofs.

	\subsection{Essential definitions} \label{subsec:defs}
		Say that $w = (w(x))_{x \in \N}$ is a \emph{weighting} of $\N$ if $w(i) \ge 0$, for all $x \in \N$ and $\sum_{x \in \N}w(x) = 1$. For an element $e \in \N^{(r)}$, define 
		\[	
			w(e):= \prod_{x \in e}w(x),
		\]
		and for a finite subset $G \subseteq \N^{(r)}$, 
		\[
			w(G):= \sum_{e \in G}w(e).
		\]
		We say that $w$ is a \emph{weighting of $[t]$} if $w$ is a weighting of $\N$ supported on $[t]$. For $G \subseteq [t]^{(r)}$, we say that $w$ is a \emph{weighting of $G$} if it is a weighting of $[t]$. We call a weighting $w$ of $\N$ \emph{maximal} for $G$ if $w(G) = \lambda(G)$. 
		Also define
		\[
			\Lambda(m,r):= \max\{\lambda(G): G \subseteq \N^{(r)}, |G| = m\}.
		\]
		So using these definitions, $\lambda(G)$ can be expressed as $\max\{w(G): w\text{ is a weighting of $\N$}\}$ and \eqref{conj:frankl-furedi} can be phrased as saying that there exists some weighting $w$ such that $w(\col{m}{r}) = \Lambda(m,r)$.  Say that a weighting $w$ of $G$ is \emph{decreasing} if $w(i) \ge w(j)$ whenever $i < j$.

		It is often convenient to work with $r$-graphs that are \emph{left-compressed}. Recall that for $1 \le x < y \le t$, the \emph{xy}-\emph{compression} of $F \in [t]^{(r)}$ is
		\[
			C_{xy}(F):= 
				\begin{cases}
					(F \setminus y) \cup x & \text{if } x \notin F, y \in F,\\
					F & \text{otherwise},
				\end{cases}
		\]
		and for $\F \subseteq [t]^{(r)}$ we define
		\[
			C_{xy}(\F):= \{C_{xy}(F): F \in \F\} \cup \{F \in \F: C_{xy}(F) \in \F\}.
		\]
		$\F$ is \emph{left-compressed} if $C_{xy}(\F) = \F$ for all $1 \le x < y \le t$.

		Let us introduce some more technical notation. For $G \subseteq [t]^{(r)}$ and $S \subseteq [t]$,
		define $N_G(S):=  \{e \setminus S: e \cup S \in G\}$; whenever $G$ is clear from the context we omit the subscript $G$. We may sometimes abuse notation and write $N(v_1,\ldots,v_s)$ when $S = \{v_1,\ldots,v_s\}$. For $x \in [t]$, define $G\setminus \{x\}$ to be the hypergraph on vertex set $[t]\setminus \{x\}$ and edge set $\{e \in G: x \notin e\}$. For vertices $x, y \in V(G)$, we define $N_y(x) := N(x) \setminus \{y\}$. 

	\subsection{Useful results from previous papers} \label{subsec:prelims-old}

		In this subsection we collect the results that we need from previous papers.

		First, we mention a few basic properties of Lagrangians proved in \cite{FraRod}.
		
		\begin{lem}[Frankl-R\"odl \cite{FraRod}]
			\label{lem:easy}
			Let $w$ be a maximal weighting of $G \subseteq [t]^{(r)}$.  Then
			\begin{enumerate}[]
				\item\label{itm:lag-w-nb} 
					For any $x \in [t]$ with $w(x) > 0$ we have 
					\[
						w(G) =  \frac{w(N(x))}{r}.
					\]
				\item\label{itm:frankl} 
					If $x,y \in [t]$ are such that $w(x), w(y) > 0$ then 
					\begin{equation} \label{eq:frankl}
						w(N(x,y)) (w(x) - w(y)) = w(N_y(x)) - w(N_x(y)).
					\end{equation} 
				\item \label{itm:lag-clique}
					$\lambda(\cl{t}{r}) = \frac{1}{t^r} \binom{t}{r}$.
			\end{enumerate}
		\end{lem}

		A useful result that we proved in \cite{Us1} tells us that if $G$ maximises the Lagrangian among $r$-graphs with at most $m \le \binom{t}{r}$ edges, then $G$ has at most $t$ vertices of non-zero weight with respect to any maximal weighting.
		  
		\begin{thm}[\cite{Us1}]\label{thm:t-vs} 
			Let $r \ge 3$ and let $t \in \N$ be sufficiently large. Suppose that $G$ is an $r$-graph with at most $\binom{t}{r}$ edges that maximises the Lagrangian among $r$-graphs of the same size, and suppose that $w$ is a decreasing maximal weighting of $G$ which satisfies $w(x) > 0$ for every $x \in V(G)$. Then $G \subseteq [t]^{(r)}$.            
		\end{thm}

		We note that a similar statement was simultaneously and independently proved by Lei and Lu~\cite{LeiLu}. In the case $r=3$, the corresponding result was proved by Talbot~\cite{Tal} and was used in \cite{Tal,TanPen,Tyo,LLP} to prove that \eqref{conj:frankl-furedi} holds for certain ranges.

		In fact, for the purpose of this paper, it is convenient to use the following corollary of \Cref{thm:t-vs}, proved in \cite{Us1}.

		\begin{cor} \label{cor:t-vs-compressed}
			Let $r \ge 3$, let $t \in \N$ be sufficiently large, and let $m$ satisfy $\binom{t}{r} - \binom{t-2}{r-2} < m \le \binom{t}{r}$. Suppose that $G$ maximises the Lagrangian among $r$-graphs with $m$ edges. Then $G$ is isomorphic to a left-compressed subgraph of $[t]^{(r)}$ with $m$ edges.
		\end{cor}

		The following observations follows immediately from \Cref{cor:t-vs-compressed}.
		\begin{cor} \label{cor:very-small-a}
			Let $r \ge 3$, let $t \in \N$ sufficiently large, let $a \in \{0,1,2\}$ and set $m:= \binom{t}{r} - a$. Suppose that $G$ is an $r$-graph that maximises the Lagrangian among $r$-graphs with $m$ edges. Then $G \simeq \col{m}{r}$.
		\end{cor}

		Finally, we quote a proposition from \cite{Us1} with some useful preliminary results. Before stating it, we need a definition.
		For a graph $G \subseteq [t]^{(r)}$ and vertex $x \in V(G)$, define $e(x)$ to be the number of edges of $\overline{G} := [t]^{(r)} \setminus G$ that contain $x$. Note that when $G$ is left-compressed we have $e(t) \ge e(t-1) \ge \ldots e(1)$. In what follows, our big-$O$ notation treats $r$ as a constant.
			
		\begin{prop} \label{prop:prelims-old} 
			Let $r \ge 3$, let $t \in \N$ be sufficiently large, let $a$ be such that $0 \le a \le \binom{t-2}{r-2}$  and set $m := \binom{t}{r} - a$.
			Suppose that $G$ is a subgraph of $[t]^{(r)}$ with at most $m$ edges, and suppose that $G$ is left-compressed and that it maximises the Lagrangian among $r$-graphs with $m$ edges. Let $w$ be a decreasing maximal weighting of $G$.
			The following statements hold.
			\begin{enumerate}[]
				\item \label{itm:useful-missing-1}
					$e(1) \le \frac{ra}{t} = O(t^{r-3})$.
				\item \label{itm:useful-w-x}
					$w(t) = w(1) - \Theta\left(\frac{e(t) - e(1)}{t^{r-1}}\right)$ and, for $x < t$, $w(x) = w(1) - \Theta\left(\frac{e(x) - e(1)}{t^{r-2}}\cdot w(t)\right)$.
			\end{enumerate}
		\end{prop}

	\subsection{Additional useful results} \label{subsec:prelims-new}
		
		In this subsection we give another proposition with further useful preliminary results, along with its proof.

		\begin{prop}\label{prop:prelims-new}
			Let $r \ge 3$, $t \in \N$ large, $a$ such that $0 \le a < \binom{t-2}{r-2}$, and set $m := \binom{t}{r} - a$.
			Let $G \subseteq [t]^{(r)}$ maximise the Lagrangian among $r$-graph with $m$ edges, and suppose that $G$ is left-compressed. Let $w$ be a decreasing maximal weighting of $G$. Then the following properties hold.
			\begin{enumerate}[]				
				\item \label{itm:useful-w-1}
					$w(1) = \frac{1}{t} + O(a t^{-r})$.
				\item \label{itm:useful-e-t}
					$e(t) = \Omega(a)$.
				\item \label{itm:useful-w-t-large}
					$w(t) \ge \frac{1}{4t}\big(1 + O(t^{-1})\big)$.
				\item \label{itm:useful-e-t-1}
					If $\frac{7}{8} \cdot \binom{t-2}{r-2} \le a < \binom{t-2}{r-2}$ then  $e(t-1) = \Omega(a)$.
			\end{enumerate}
		\end{prop}
			
		\begin{proof}
			Note that 
			\[	
				\lambda(G) \ge \left(\binom{t}{r} - a\right)\frac{1}{t^r},
			\]
			as this is the weight of $G$ with respect to the uniform weight on $[t]$. Moreover, using \Cref{lem:easy} \ref{itm:lag-w-nb} and \ref{itm:lag-clique} and observing that $N(1)$ is a subgraph of the complete graph $\cl{t-1}{r-1}$, we have
			\[	
				\lambda(G) = \frac{w(N(1))}{r} \le \frac{1}{r}\cdot \big(1 - w(1)\big)^{r-1} \binom{t-1}{r-1}\cdot \frac{1}{(t-1)^{r-1}}.
			\]
			Putting the two inequalities together, we find that
			\begin{equation} \label{eqn:weight-1}
				(1 - w(1))^{r-1} 
				\ge  r\cdot \frac{ (t-1)^{r-1} }{\binom{t-1}{r-1}} \cdot \left(\binom{t}{r} - a \right)\frac{1}{t^r} 
				=  \left( \frac{t-1}{t} \right)^{r-1} \left(1 - \frac{a}{\binom{t}{r}}\right).
			\end{equation}
			Hence,
			\[	
				w(1) \le 1 - \left(1 - \frac{1}{t}\right)\left(1 - \frac{a}{\binom{t}{r}}\right)^{1/(r-1)} = \frac{1}{t} + O\left(a t^{-r}\right).
			\]
			This completes the proof of \ref{itm:useful-w-1}.				
		
			Now we prove a couple of claims about the Lagrangians of colex graphs.
			\begin{claim} \label{claim:useful-lower-H-i-one}
				For $0 \le b \le \binom{t-2}{r-2}$, let $H_b := \col{\binom{t}{r} - \binom{t-2}{r-2} + b}{r}$. Then 
				\[
					\lambda(H_b) \ge \lambda(\cl{t-1}{r}) + \frac{b}{4(t-1)^r}\,.
				\]
			\end{claim}
		
			\begin{proof}
				Let $w'$ be the weighting defined by $w'(t) = w'(t-1) = \frac{1}{2(t-1)}$ and $w'(x) = \frac{1}{t-1}$ for every $x \in [t-2]$. Then 
				\[
					\lambda(H_b) \ge w'(H_b) 
					= \overline{w}(\cl{t-1}{r}) + \frac{b}{4(t-1)^2} 
					= \lambda(\cl{t-1}{r}) + \frac{b}{4(t-1)^2},
				\]
				where $\overline{w}$ is the uniform weighting on $[t-1]$. This completes the proof of the claim.
			\end{proof}

			\begin{claim} \label{claim:useful-lower-H-i-two}
				For $0 \le a \le \binom{t-2}{r-2}$, let $F_a := \col{\binom{t}{r} - a}{r}$. Then 
				\[
					\lambda(F_a) 
					\ge \lambda(\cl{t}{r}) - \frac{a}{t^r} + \Omega\left(\frac{a^2}{t^{2(r-1)}}\right).
				\]
			\end{claim}

			\begin{proof}
				Note that $F_{a+1}$ can be obtained by removing one edge from $F_{a}$; denote this edge by $f_a$, and let $w_a$ be a weighting of $[t]$ such that $w_a(F_a) = \lambda(F_a)$. Since $e(t) = a$ in $F_a$ (i.e.\ all non-edges contain $t$), it follows from \ref{itm:useful-w-1} and \Cref{prop:prelims-old} \ref{itm:useful-missing-1} and \ref{itm:useful-w-x} that $w_a(t) = w_a(1) - \Omega(a t^{-(r-1)}) = t^{-1} - \Omega(a t^{-(r-1)})$ and thus $w_a(f_a) \le w(1)^{r-1} \cdot w(t) \le t^{-r} - \Omega(a t^{-2(r-1)})$.
				Therefore
				\[	
					\lambda(F_{a+1}) 
					\ge w_a(F_{a+1}) 
					= w_a(F_a) - w_a(f_a) 
					= \lambda(F_a) - \frac{1}{t^{r}} + \Omega\!\left(\frac{a}{t^{2(r-1)}}\right).
				\]
				Hence 
				\begin{align*}
					\lambda(\cl{t}{r}) - \lambda(F_a) 
					= \sum_{b=0}^{a-1}\! \left(\lambda(F_b) - \lambda(F_{b+1})\right)  \le \sum_{b=0}^{a-1} \left(\frac{1}{t^r} - \Omega(\frac{b}{t^{2(r-1)}})\right)
					= \frac{a}{t^{r}} - \Omega\!\left(\frac{a^2}{t^{2(r-1)}}\right).
				\end{align*}
				It follows that $\lambda(F_a) \ge \lambda(\cl{t}{r}) - \frac{a}{t^r} + \Omega(\frac{a^2}{t^{2(r-1)}})$, as required.
			\end{proof}
				
			We now prove \ref{itm:useful-e-t}. As the weight of each non-edge of $G$ is at least $w(t)^r$,
				\begin{align*}
					\lambda(\cl{t}{r}) \ge w(\cl{t}{r}) = w(G) + w(\overline{G}) \ge w(G) + a \cdot w(t)^r.
				\end{align*}
			On the other hand, as $G$ maximises the Lagrangian and by \Cref{claim:useful-lower-H-i-two},			
			\begin{align*}
				w(G) 
				\ge \lambda(H_a)
				\ge \lambda(\cl{t}{r}) - \frac{a}{t^r} 
				+ \Omega\left(\frac{a^2}{t^{2(r-1)}}\right).
			\end{align*}
			It follows that $w(t)^r \le t^{-r} - \Omega(a t^{-2(r-1)})$, so $w(t) \le t^{-1} - \Omega(a t^{-(r-1)}) \le w(1) - \Omega(a t^{-(r-1)})$, using $w(1) \ge t^{-1}$. By \Cref{prop:prelims-old} \ref{itm:useful-w-x} we have $w(1) - w(t) = O(\frac{e(t) -e(1)}{t^{r-1}})$. Hence $e(t) = \Omega(a)$ as needed for \ref{itm:useful-e-t}.
				
			Now we prove \ref{itm:useful-w-t-large}. Write $e(G) = \binom{t}{r} - \binom{t-2}{r-2} + b$ (so $b \ge 1$). Note that 
			\begin{equation} \label{eqn:lag-G-wrt-i}
				\lambda(G) \ge \lambda(F_b) \ge \lambda(\cl{t-1}{r}) + \frac{b}{4(t-1)^r},
			\end{equation}
			where the first inequality holds as $G$ maximises the Lagrangian, and the second inequality follows from \Cref{claim:useful-lower-H-i-one}.
			Let $G'$ be a graph obtained from $G$ by removing any $b$ edges that contain both $t$ and $t-1$ (note that such edges exist). By \ref{itm:useful-w-1}, we have $w(1) \le \frac{1}{t}\big(1 + O(t^{-1})\big)$, hence the weight of the edges removed is at most 
			\[	
				b \cdot w(t)w(t-1)w(1)^{r-2} 
				\le b \cdot w(t)w(t-1) \cdot \frac{1}{t^{r-2}}\big(1 + O(t^{-1})\big),
			\]
			and at least 
			\[
				w(G) - w(G') \ge w(G) - \lambda(\cl{t-1}{r}) 
				\ge \frac{b}{4(t-1)^r},
			\]
			by \Cref{thm:main1} and \eqref{eqn:lag-G-wrt-i}. Combining these inequalities gives that $w(t)w(t-1) \ge \frac{1}{4t^2}\big(1 + O(t^{-1})\big)$. In particular, $w(t) \ge \frac{1}{4t}\big(1 + O(t^{-1})\big)$, as required.

			Finally we turn our attention to \ref{itm:useful-e-t-1}. By  \ref{itm:useful-w-1}-\ref{itm:useful-w-t-large} and \Cref{prop:prelims-old} \ref{itm:useful-missing-1} and \ref{itm:useful-w-x},  $w(t) \le \frac{1 - \delta}{t}$, for some constant $\delta > 0$, and we may assume that $\delta < 1/10$. If $w(t-1) \le \frac{1 - \delta/r}{t}$, then by \Cref{prop:prelims-old} \ref{itm:useful-w-x} we have $e(t-1) = \Omega(a)$, as required. So suppose otherwise. By \Cref{lem:easy} \ref{itm:frankl}, we have
			\begin{equation} \label{eqn:compare-w-t-1}
				w(1) - w(t) = \frac{w(N_t(1)) - w(N_1(t))}{w(N(1, t))} \ge (e(t) - e(1)) \cdot w(t-1)^{r-1} \cdot (r-2)!,
			\end{equation}
			where the first inequality follows since $w(t-1)^{r-1}$ is a lower bound on the weight of any $(r-1)$-tuple in $[t] \setminus \{1, t\}$ and $w(N(1, t)) \le \lambda(\cl{t-2}{r-2}) \le \frac{1}{(r-2)!}$. 
			Since an $r$-tuple that does not contain $t$ has weight at least $\frac{(1 - \delta/r)^r}{t^r} \ge w(t) \cdot w(1)^{r-1}$, every non-edge contains $t$, i.e.\ $e(t) = a$. So
			\begin{align*}
				w(1) - w(t) 
				&\ge a \cdot  \frac{(1 - \delta/r)^{r-1}}{t^{r-1}} \cdot (r-2)! \cdot \left(1 + O(t^{-1})\right)
				\ge \frac{9a}{10}\cdot \frac{1}{t^{r-1}} \cdot (r-2)! ,
			\end{align*}
			where the first inequality follows from \ref{itm:useful-missing-1}, the assumptions on $w(t-1)$ and $e(t)$ and \eqref{eqn:compare-w-t-1}, and the second inequality follows since $\delta < 1/10$.
			Since  $w(1) \le \frac{1}{t}\left(1 + O(t^{-1})\right)$ and $w(t) \ge \frac{1}{4t}\left(1 + O(t^{-1})\right)$ (by \ref{itm:useful-w-1} and \ref{itm:useful-w-t-large}), it follows that
			\begin{align*}
				a \le \frac{3}{4}\cdot \frac{10}{9} \cdot \frac{t^{r-2}}{(r-2)!} \cdot (1 + O(t^{-1})) < \frac{7}{8} \binom{t-2}{r-2},
			\end{align*}
			a contradiction. Thus $e(t-1) = \Omega(a)$, as required for \ref{itm:useful-e-t-1}.
		\end{proof}

\subsection{Proof of \Cref{prop:i-1-intersecting-families} -- maximising sum of degrees squared} \label{sec:sum-degs-squared}

	\begin{proof}[Proof of \Cref{prop:i-1-intersecting-families}]
		First, we claim that $|e \cap f| = r - 1$ for every pair of distinct edges $e,f \in E(H)$. Indeed, let $m = e(H)$. Let $H'$ be an $r$-graph that consists of $m$ edges that all contain a fixed set of $r-1$ vertices. Then $P_2(H') = (r-1)m^2 + m$ (as the vertices in $S$ have degree $m$ and the remaining $m$ non-isolated vertices have degree $1$). Hence, $P_2(H) \ge (r-1)m^2 + m$. Also,
		$$
			P_2(H)  = \sum_{e \in E(H)} \sum_{x \in e} d(x)  = \sum_{e, f \in E(H)} |e \cap f| \le \sum_{e \in E(H)} (r + (m-1) \cdot (r-1)) \le (r-1)m^2 + m.
		$$
		Here, the first inequality follows as for any edge $e$, for every other edge $f$, we have $|e \cap f| \le r-1$. In fact, since $P_2(H) \ge (r-1)m^2 + m$, we must have equality. Hence, $|e \cap f| = r-1$ for every distinct $e, f \in E(H)$, as desired. 

		Let $e, f \in E(H)$ be distinct; so $|e \cap f| = r-1$. If $e \cap f \subseteq g$ for every $g \in E(H)$, then all edges of $H$ contain the set $e \cap f$, of size $r-1$, as claimed. So let us assume that there is an edge $g$ such that $e \cap f \nsubseteq g$. We claim that $h \subseteq e \cup f$ for every $h \in E(H)$. We first show this for $h = g$. Write $S = g \cap (e \cup f)$; if $g \nsubseteq e \cup f$ then $|S| \le r-1$ and $|S \cap e|, |S \cap f| \ge r-1$, which implies that $S = e \cap f$, a contradiction to the choice of $g$. Now let $h \in E(H)$. Suppose that $h \nsubseteq e \cup f$, and denote $T = h \cap (e \cup f)$. Then $|T| \le r-1$ and $|T \cap e|, |T \cap f|, |T \cap g| \ge r-1$. It follows that $|e \cap f \cap g| \ge r-1$, a contradiction. Hence, all edges of $H$ are contained in $e \cup f$, a set of size $r+1$, so $H$ is a subgraph of the clique $(e \cup f)^{(r)}$, where $|e \cup f| = r+1$, as claimed.
	\end{proof}

\section{Dealing with large $a$} \label{sec:a-med}

	The main theorem of the section is the following.
		
	\begin{thm} \label{thm:non-edges-big-range}
		Let $r \ge 3$ and $2 \le i \le r-1$, let $t \in \N$ be sufficiently large, and let $a$ be such that $t^{0.01} \cdot \binom{t-(i+1)}{r-(i+1)} \le a \le \binom{t-i}{r-i}$ and $a \neq \binom{t-2}{r-2}$. Suppose that $G$  maximises the Lagrangian among $r$-graphs with at most $m = \binom{t}{r} - a$ edges. Then $G$ is isomorphic to a left-compressed subgraph of $[t]^{(r)}$ with $m$ edges, all of whose non-edges contain $\{t - (i-1), \ldots, t\}$.
	\end{thm}

	\begin{rem}
		We believe that we can extend this to all $a > (r+1) \binom{t - (i+1)}{r - (i+1)}$ (using a more careful analysis of the arguments in the proof of \Cref{thm:non-edges-big-range} as well as a different argument for very small $a$ that uses \Cref{lem:eval-lag}). However, we did not wish to make the paper more complicated, so we do not include this proof here.
	\end{rem}
	
	It is not difficult to see that \Cref{thm:non-edges-big-range} implies \Cref{thm:range} when $a \ge t^{0.01}$.

	\begin{proof}[Proof of \Cref{thm:range} when $a \ge t^{0.01}$]
		By \Cref{thm:non-edges-big-range}, we may assume that $G$ is a subgraph of $[t]^{(r)}$ on $m$ edges all of whose non-edges contain $\{t-(r-2), \ldots, t\}$. But this determines $G$ uniquely and implies that $G \simeq \col{m}{r}$, as required.
	\end{proof}
	
	The rest of the section is devoted to the proof of \Cref{thm:non-edges-big-range}. Say that $x \not= y \in V(G)$ are \emph{twins} if $N_x(y) = N_y(x)$. 
	
	As in the proof of \Cref{thm:main1} in \cite{Us1}, we will compare $G$ with a graph $H$ on the same number of edges in which all non-edges contain $\{t-(i-1),\ldots,t\}$. As $G$ and $H$ have the same number of edges, we think of the members of $E(G)\setminus E(H)$ as being paired with the members of $E(H)\setminus E(G)$. We think of $H$ as being obtained from $G$ by swapping the edges in each of these pairs. We first find an upper bound on the weight lost when replacing $G$ with $H$, and later we show that a modification of $w$ allows us to gain more weight than we lost, thus reaching a contradiction. For technical reasons,  we require that $1$ and $t-i$ are twins in $H$. 
	
	\begin{proof}[Proof of \Cref{thm:non-edges-big-range}]
		By \Cref{cor:t-vs-compressed}, we may assume that $G$ is a left-compressed subgraph of $[t]^{(r)}$ on $m$ edges. Let $w$ be a decreasing maximal weighting of $G$.
		Let $I := \{t-(i-1), \ldots, t\}$. Suppose that there exists an edge of $\overline{G} := [t]^{(r)} \setminus G$ that does not contain $I$. 
				
		\begin{claim}\label{Htwins}
			There exists an $r$-graph $H$ on vertex set $[t]$ with $|H| = |G|$ such that: every edge of $\overline{H}$ contains $I$; the vertices $1$ and $t-i$ are twins in $H$, and; all but $O(t^{r-i-1})$ $r$-tuples in $E(H) \setminus E(G)$ do not contain $I$.
		\end{claim}
		\begin{proof}
			Let $F$ be a graph obtained from $G$ by swapping each edge of $\overline{G}$ that does not contain $I$ with an edge of $G$ that does; note that such a graph exists by our assumption on the number of non-edges. So every edge of $\overline{F}$ contains $I$. 
			
			We will show that there exists a graph $H$ with $|H| = |F|$ in which the vertices $1$ and $t-i$ are twins, and $E(H) \triangle E(F)$ contains only $r$-tuples that contain $I$ and at least one of $1$ and $t-i$. This suffices to prove the claim as the condition on $E(H) \triangle E(F)$ ensures both that every edge of $\overline{H}$ contains $I$, and that all but $O(t^{r-i-1})$ $r$-tuples in $E(H) \setminus E(G)$ do not contain $I$ (as $\left|E(H) \triangle E(F)\right| = O(t^{r-i-1})$). 
			
			Let $\A$ and $\B$ be defined as follows.
			\begin{align*}
				& \A = \{A \subseteq [t], |A| = r-1 : A \cup \{1\} \in E(F) \text{ and } A \cup \{t-i\} \notin E(F) \} \\
				& \B = \{B \subseteq [t], |B| = r-1 : B \cup \{t-i\} \in E(F) \text{ and } B \cup \{1\} \notin E(F) \}.
			\end{align*}
			Denote $a = |\A|$ and $b = |\B|$. If $a$ and $b$ have the same parity, fix any subset $\S \subseteq \A \cup \B$ of size $(a+b)/2$, and define
				\begin{align*}
					& T_1:= \{A \cup \{1\}: A \in \A\} \cup \{B \cup \{t-i\}: B \in \B\},\\ 
					& T_2:= \{C \cup \{1\}: C \in \S\} \cup \{C \cup \{t-i\}: C \in \S\}.
				\end{align*}
				
			If $a$ and $b$ have different parities, we consider two cases. If there is an edge $e$ in $F$ that contains $I \cup \{1, t-i\}$, then let $\S$ be a subset of $\A \cup \B$ of size $(a+b+1)/2$, and define
				\begin{align*}
					& T_1:= \{A \cup \{1\}: A \in \A\}\cup \{B \cup \{t-i\}: B \in \B\} \cup \{e\},\\
					& T_2:= \{C \cup \{1\}: C \in \S\} \cup\{C \cup \{t-i\}: C \in \S\}.
				\end{align*}
			Otherwise, if there are no edges that contain $I \cup \{1, t-i\}$, then let $\S \subseteq \A \cup \B$ have size $(a+b-1)/2$ and let $e$ be any $r$-set containing $I \cup \{1, t-i\}$. Set
			\begin{align*}
				& T_1:= \{A \cup \{1\}: A \in \A\} \cup \{B \cup \{t-i\}: B \in \B\},\\
				& T_2:= \{C \cup \{1\}: C \in \S\} \cup \{C \cup \{t-i\}: C \in \S\} \cup \{e\}.
			\end{align*} 
			For each case, define $H$ to be the graph obtained from $F$ by replacing the edges $T_1$ by $T_2$. Note that $|T_1| = |T_2|$ and so $|H| = |F| = |G|$. By definition, the vertices $1$ and $t-i$ are twins in $H$, and every member of $E(H) \triangle E(F)$ contains $I$ and at least one of $\{1, t-i\}$. This completes the proof of the claim.
		\end{proof}

		We now bound the weight lost by replacing $G$ with $H$.
		
		\begin{claim} \label{cl:weight-lost-general-case}
			Let $\gamma:= 0.001$. Then $w(H) - w(G) = O(t^{-\gamma/r} \cdot \frac{a^2}{t^{2(r-1)}})$.
		\end{claim}

		\begin{proof}
			We pair the edges of $E(G) \setminus E(H)$ with those of $E(H) \setminus E(G)$, so that we consider $H$ as being obtained from $G$ by a series of swaps, such that in all but $O(t^{r-i-1})$ swaps an edge in $G$ that contains $I$ is swapped with a non-edge that does not contain $I$. Furthermore, the remaining swapped pairs consist of two $r$-tuples that contain $I$.

			Let $f := (x_1, \ldots, x_r)$ be a non-edge in $G$ (that may or may not contain $I$) and suppose that it is swapped with the edge $g := (y_1, \ldots, y_{r-i}, t-(i-1), \ldots, t)$ in $G$. Then
			\begin{align}\label{eq:edgebound}
				w(f) - w(g) \le \left(w(1)^{r-i} - w(x_{r-i})^{r-i}\right)\cdot \prod_{j=0}^{i-1}w(t-j) 
				=  O\left(\frac{e(x_{r-i})}{t^{2(r-1)}}\right).
			\end{align}
			Here, we used the fact that $w(1)^s - w(x)^s = O\left( (w(1) - w(x)) \cdot  t^{-(s-1)} \right) = O(\frac{e(x)}{t^{r+s-2}})$ for every $x \in [t]$ and $s \le r$, which follows from \Cref{prop:prelims-old} \ref{itm:useful-w-x} and the fact that $w(1) = O(t^{-1})$ (see \Cref{prop:prelims-new} \ref{itm:useful-w-1}). 

			First consider the contribution from pairs consisting of an edge and a non-edge that both contain $I$. Since there are $O(t^{r-i-1}) = O(t^{-0.01} a)$ of them, their contribution is $O(\frac{t^{-0.01}a^2}{t^{2(r-1)}})$.
			
			Next consider the remaining swapped pairs. We distinguish two cases.
			If $e(t-(i-1)) \le t^{-\gamma/r} \cdot a$, then $e(x_{r-i}) \le t^{-\gamma/r} \cdot a$, which implies that the loss from each swap is $O(t^{-\gamma/r}\cdot \frac{a}{t^{2(r-1)}})$, and $O(t^{-\gamma/r}\cdot \frac{a^2}{t^{2(r-1)}})$ in total, as required.
			
			Otherwise, i.e.\  when $e(t-(i-1)) \ge t^{-\gamma/r} \cdot a$, we define $S = \{x \in [t] : e(x) \ge t^{-2\gamma/r} \cdot a \}$. Then $|S| = O(t^{2\gamma/r})$, as the total number of missing edges is $a$. We show that every non-edge contains at least $i$ vertices of $S$. Suppose $x:=(x_1,\ldots,x_r)$ is a non-edge that contains at most $(i-1)$ vertices of $S$. By definition of $S$, each vertex $u \notin S$ satisfies $e(u) \le t^{-2\gamma/r} \cdot a \le e(t - (i-1))\cdot t^{-\gamma/r} $.   
			Then using \Cref{prop:prelims-old} \ref{itm:useful-missing-1} and \ref{itm:useful-w-x}, we have
			\begin{align}\label{eq:annoying-calc}
				\begin{split}
				w(x) 
				& \ge  w(t) \cdots w(t-(i-2)) \cdot \left(w(1) - O\left(\frac{e(t- (i-1))\cdot t^{-\gamma/r} }{t^{r-1}}\right)\right)^{r-(i-1)}\\
				& \ge  w(t) \cdots w(t-(i-2))\cdot \left(w(1) - o\left(\frac{e(t- (i-1)) }{t^{r-1}}\right)\right) \cdot w(1)^{r-(i-2)} \\
				&> w(t) \ldots w(t-(i-1))\cdot w(1)^{r-(i-2)}.
				\end{split}
			\end{align}
			So swapping $x$ with any edge of $G$ containing $I$ would increase the weight of $G$, a contradiction.				
			
			Given this, we consider two types of non-edges: those that contain exactly $i$ vertices of $S$, and those that contain at least $i+1$ vertices of $S$. The number of non-edges of the second type is at most 
			$$O(|S|^{i+1}) \cdot \binom{t-(i+1)}{r-(i+1)} = O(t^{(2\gamma/r) \cdot (i+1) + r - (i+1)}) = O(a \cdot t^{-8\gamma}),$$
			hence (using \eqref{eq:edgebound}) the loss from swaps involving non-edges of the second type is $O(\frac{t^{-8\gamma} \cdot a^2}{t^{2(r-1)}})$. The loss per swap that involves an edge of the first type is $O(\frac{t^{-2\gamma/r} \cdot a}{t^{2(r-1)}})$, hence (again using \eqref{eq:edgebound}) in total $O(\frac{t^{-2\gamma/r} \cdot a^2}{t^{2(r-1)}})$. The claim follows.
		\end{proof}
		
		We now show that by changing the weight $w$ of the vertices of $H$ slightly, we can obtain a graph whose weight is larger than the weight of $G$, thus reaching a contradiction to the choice of $G$ and $w$.
		
		\begin{claim} \label{cl:pick-two-vs}
			Let $\gamma = 0.001$. Then there exist two vertices $x, y \in [t]$ that are twins in $H$, such that $|N_H(x,y)| = \Omega(t^{r-2})$ and $w(x) - w(y) = \Omega(t^{-\gamma/10r} \cdot \frac{a}{t^{r-1}})$.
		\end{claim}

		\begin{proof}
			We consider two cases. First suppose that $e(t-(i-1)) \le t^{-\gamma/20r} \cdot a$. We will show that $x = t - (i-1)$ and $y = t$ satisfy the properties of the claim. First, note that by choice of $H$, these vertices are twins in $H$. Now observe that if $i \ge 3$, then $a \le \binom{t-3}{r-3} = O(t^{r-3})$. But when $i=2$, if $\frac{7}{8} \binom{t-2}{r-2} \le a$, then by hypothesis and \Cref{prop:prelims-new} \ref{itm:useful-e-t-1}, we would have $e(t-1) = \Omega(a)$, a contradiction. So $a \le \frac{7}{8} \binom{t-2}{r-2}$. So, as by definition of $H$ every edge in $\overline{H}$ contains $x$ and $y$ and $|\overline{H}| = a$, we have $|N_H(x,y)| \ge \frac{1}{8}\binom{t-2}{r-2}$.
			
			We will now show that $w(t-(i-1)) - w(t) = \Omega\left(\frac{a}{t^{r-1}}\right)$. Observe that $w(N_G(t-(i-1),t))= \Omega(1)$. Indeed, as $|N_G(t-(i-1),t)| \ge |N_H(t-(i-1),t)|$, at least $\frac{1}{8} \binom{t-2}{r-2}$ $(r-2)$-tuples are present in $N_G(t-(i-1),t)$. By \Cref{prop:prelims-new} \ref{itm:useful-w-t-large} (and using that $w$ is decreasing) each of them has weight $\Omega(t^{-(r-2)})$. 
			
			So using this and \Cref{lem:easy} \ref{itm:frankl} applied to $w$ as a weighting of $G$, we have 
				\begin{align*}
					w(t-(i-1)) - w(t) &= (w(N_t(t-(i-1))) - w(N_{t-(i-1)}(t)))/w(N_G(t-(i-1),t)) \\ &= \Omega\big(w(N_{t}(t-(i-1)) \setminus N_{t-(i-1)}(t))\big).
				\end{align*}
			As $G$ is left-compressed, 	$N(t) \subseteq N(t-(i-1))$ and so this is at least
				$$\Omega\left(\frac{e(t)-e(t-(i-1))}{t^{r-1}} \right) 
				=  \Omega\left(\frac{a}{t^{r-1}}\right),$$ 
			where we used $w(t) = \Omega(t^{-1})$ by \Cref{prop:prelims-new} \ref{itm:useful-w-t-large}, $e(t - (i-1)) \le t^{-\gamma/20r}\cdot a$ by assumption and \Cref{prop:prelims-new} \ref{itm:useful-e-t}.

			Now suppose that $e(t-(i-1)) \ge t^{-\gamma/20r} \cdot a$. Suppose, in order to obtain a contradiction, that $e(t-i)  \le t^{-\gamma/10r} \cdot a \le e(t - (i-1)) t^{-\gamma/10r}$. Then (similarly to the calculation of \eqref{eq:annoying-calc}), if $u = (u_1,\ldots,u_r)$ is a non-edge of $G$ containing at most $(i-1)$ vertices of $I$, then using \Cref{prop:prelims-old} \ref{itm:useful-missing-1} and \ref{itm:useful-w-t-large}  we have
				\begin{align*}
					w(u) &\ge w(t)w(t-1)\cdots w(t-(i-2))\left(w(1) - O\left(\frac{e(t-i) - e(1)}{t^{r-2}}\right)w(t)\right)^{r-(i-1)}\\
					&\ge w(t)w(t-1)\cdots w(t-(i-2))\left(w(1) - o\left(\frac{e(t-(i-1))}{t^{r-2}}\right)w(t)\right)w(1)^{r-(i-2)}\\
					&>  w(t)\cdots w(t-(i-1))w(1)^{r-(i-2)}.
				\end{align*}
			So replacing $u$ with any non-edge containing $I$ would increase the weight of $G$. Thus, all non-edges of $G$ contain $I$, a contradiction. So $e(t-i) = \Omega\left(t^{-\gamma/10r} \cdot a\right)$.
			
			Hence $w(1) - w(t-i) = \Omega(\frac{t^{-\gamma/10r}a}{t^{r-1}})$ by \Cref{prop:prelims-old} \ref{itm:useful-missing-1} and \ref{itm:useful-w-x}. Then we may take $x = 1, y = t-i$; indeed, by Claim~\ref{Htwins} we have that $1$ and $t-i$ are twins in $H$, and $|N_H(1, t-i)| = \binom{t-2}{r-2} - O(t^{r-i-2})$ (because all non-edges contain $I$).
		\end{proof}

		Let $x, y$ be as in \Cref{cl:pick-two-vs}. Define $w'$ as follows.
		\begin{align*}
			w'(z) = \left\{
				\begin{array}{ll}
					w(z) & z \neq x, y \\
					\frac{1}{2} \cdot \left(w(x)) + w(y)\right) & z \in \{x, y\}.
				\end{array}
			\right.
		\end{align*}
		Note that $w'$ is a legal weighting. Since $x$ and $y$ are twins in $H$, the contribution of edges that contain none or exactly one of them to the weight of $H$ is the same in $w$ and in $w'$. We thus have the following.
		\begin{align}\label{eqn:weight-gained-case-1}
		\begin{split}
			w'(H) - w(H) 
			& = w(N_H(x, y)) \cdot \left(\left(\frac{w(x) + w(y)}{2}\right)^2 - w(x) w(y)\right) \\
			& = w(N_H(x, y)) \cdot \left(\frac{w(x) - w(y)}{2}\right)^2 \\
			& = \Omega \left(\frac{t^{-\gamma/20r} a^2}{t^{2(r-1)}}\right),
		\end{split}
		\end{align}
		where the last equality follows as $w(N(x,y)) = \Omega(1)$ (as $|N(x,y)| = \Omega(t^{r-2})$), and each edge in $N(x,y)$ has weight $\Omega(t^{-(r-2)})$) and by the assumption on $x$ and $y$. It follows from \Cref{cl:weight-lost-general-case} that $w(G) < w'(H)$, contradicting the choice of $G$ and $w$. This completes the proof of \Cref{thm:non-edges-big-range}.
	\end{proof}

\section{Understanding the structure when $a$ is small}\label{sec:a-small}

	The main result of the section is the following lemma, which evaluates the Lagrangian of $G$ when $G$ is a left-compressed subgraph of $[t]^{(r)}$ with not too many non-edges.

	\begin{lem} \label{lem:eval-lag}
		Let $r \ge 3$, let $t \in \N$, and let $a \le \binom{t-2}{r-2}$. Suppose that $G$ is a subgraph of $[t]^{(r)}$ with $\binom{t}{r} - a$ edges and that $G$ is left-compressed. Then
		\begin{equation}
		\label{eq:lagsq}
			\lambda(G) = \mu_0 - \frac{a}{t^r} + \frac{1}{2\mu_2 t^{2(r-1)}} \cdot \sum_x e(x)^2 - \frac{r^2a^2}{2\mu_2 t^{2r-1}} + O\left(a^3 t^{-3r+4}\right),
		\end{equation}
		where $\mu_i = \binom{t-i}{r-i}\frac{1}{t^{r-i}}$ and $e(x)$ is the number of non-edges incident with vertex $x$.
	\end{lem}
	Note that this bound becomes effective when $a = o(t^{r-2})$, as then the error term is smaller than the third term.

	Before proving \Cref{lem:eval-lag} we will state some corollaries. It may be helpful to recall the definitions given in \Cref{subsec:sqs}. The main conclusion that we draw from \Cref{lem:eval-lag} is the following immediate corollary. Recall that $\overline{G} := [t]^{(r)} \setminus G$). 

	\begin{cor} \label{cor:complement-max-degrees-squared}
		Let $r \ge 3$, $t \in \N$ be large, $a < \binom{t-2}{r-2}$ and $m := \binom{t}{r} - a$.
		Suppose that $G$ maximises the Lagrangian among $r$-graphs with $m$ edges. Then, up to relabelling of the vertices, $G$ is a left-compressed subgraph of $\cl{t}{r}$, and
		\[
			P_2(\overline{G}) = (1 + O(a t^{-(r-2)})) \cdot P_2(r,a,t).
		\]
		In particular, if $a = o(t^{(r-2)/3})$ then $P_2(\overline{G}) = P_2(r,a,t).$ 
	\end{cor}
	
	\begin{proof}
		By \Cref{cor:t-vs-compressed}, $G$ is isomorphic to a left-compressed subgraph of $[t]^{(r)}$ with $m$ edges. Observe that $|E(\overline{G})| = a$, $\sum_{x \in G} e(x)^2 = P_2(\overline{G})$ and $P_2(\overline{G}) = O(a^2)$. As $G$ maximises the Lagrangian amongst $r$-graphs with $m$ edges, \Cref{eq:lagsq} takes its maximal value over all $r$-graphs with $m$-edges. This occurs when
			\[
				P_2(\overline{G}) = (1 + O(a t^{-(r-2)})) \cdot P_2(r,a,t).\qedhere
			\]
	\end{proof}
	Using this, we can complete the proof of \Cref{thm:range} via the following corollary.
	\begin{cor} 
		Let $r \ge 3$, let $t \in \N$ be large, let $a$ be such that $a \notin \{3, \ldots, r+1\}$ and $a = o(t^{1/3})$, and let $m := \binom{t}{r} - a$. Suppose that $G$ maximises the Lagrangian among $r$-graphs with $m$ edges. Then $G \simeq \col{|G|}{r}$.
	\end{cor}

	\begin{proof}
		By \Cref{cor:t-vs-compressed}, we may assume that $G \subseteq [t]^{(r)}$ and that $G$ is left-compressed. Then using  \Cref{cor:complement-max-degrees-squared}  as $a = o(t^{1/3})$, we have that $P_2(\overline{G}) = P_2(r,a,t) = P_2(r,|G|)$. Now by \Cref{prop:i-1-intersecting-families}, $\overline{G}$ is a \emph{star}, i.e.\ all its edges contain a fixed set of $r-1$ vertices, namely $\{t-(r-2), \ldots, t\}$ (as $G$ is left-compressed). This determines $G$ uniquely and implies that $G \simeq \col{|G|}{r}$.
	\end{proof}
		The remainder of the section is devoted to proving \Cref{lem:eval-lag}. Our plan is as follows: we first obtain good estimates for $w(x)$ in terms of $e(x)$ (the number of edges of $\overline{G}$ incident with $x$), and then we estimate the weight of the non-edges (i.e.\ we estimate $w(\overline{G}) = w(\cl{t}{r}) - w(G)$), and also the difference between $w(\cl{t}{r})$ and $\lambda(\cl{t}{r})$. Putting these two estimates together, we obtain an estimate for $w(G)$.
		
	\begin{proof}[Proof of \Cref{lem:eval-lag}]
		Let $w$ be a decreasing weighting of $G$ such that $w(G) = \lambda(G)$. Let $\alpha$ and $\delta(x)$, for $x \in [t]$, be the unique values satisfying $w(x) = \alpha - \delta(x)$, and  
		\begin{align} \label{eqn:def-delta-1}
			\mu_i := \binom{t-i}{r-i}t^{-(r-i)} \qquad \qquad 
			\delta(1) := \frac{e(1)}{\mu_2 t^{r-1}}.
		\end{align}
		Note that
		$
			1 = \sum_x w(x) = \sum_x (\alpha - \delta(x)) = t \alpha - \sum_x \delta(x),
		$
		which implies that
		\begin{equation} \label{eqn:alpha}
			\alpha = \frac{1}{t}\cdot \left(1 + \sum_x \delta(x)\right).
		\end{equation}
		\begin{claim} \label{cl:eval-alpha-delta}
			The following estimates hold.
			\begin{enumerate}[]
				\item \label{itm:cl-alpha}
					$\alpha = t^{-1}\cdot\left(1 + O(a t^{-(r-1)})\right)$.
				\item \label{itm:cl-sum-delta}
					$\sum_x \delta(x) = O(a t^{-(r-1)})$.
				\item \label{itm:cl-w}
					$w(x_1) \cdots w(x_s) = t^{-s}\cdot \left(1 + O(a t^{-(r-2)})\right)$ for every $x_1, \ldots, x_s \in [t]$, where $s \le r$.
				\item \label{itm:cl-delta}
					$\delta(x) = \frac{e(x)}{\mu_2 t^{r-1}}\cdot \left(1 + O(a t^{-(r-2)})\right)$ for all $x \in [t]$.
			\end{enumerate}
		\end{claim}
		\begin{proof}
			By \Cref{prop:prelims-new} \ref{itm:useful-w-1},  we have $w(1) = t^{-1} + O(at^{-r})$, and, since $e(1) \le ar/t$ (by \Cref{prop:prelims-old} \ref{itm:useful-missing-1}), we have $\delta(1) = O(a t^{-r})$. As $\alpha = w(1) + \delta(1)$, it follows that $\alpha = t^{-1}\left(1 + O(a t^{-(r-1)})\right)$, as required for \ref{itm:cl-alpha}. 
			
			Thus, by \eqref{eqn:alpha}, $\sum_x \delta(x) = O(at^{-(r-1)})$, establishing \ref{itm:cl-sum-delta}. 

			Let $\{x_1, \ldots, x_s\}$ be a subset of $[t]$, where $s \le r$. Then 
			\begin{align*} 
			\begin{split}
				w(x_1) \cdots w(x_s) 
				& = (\alpha - \delta(x_1)) \cdots (\alpha - \delta(x_s)) \\
				& = \alpha^s - \alpha^{s-1} \sum_j \delta(x_j) + O\left( t^{-(s-2)} \left(a t^{-(r-1)}\right)^2 \right) \\
				& = t^{-s}\left(1 + O(a t^{-(r-2)})\right),
			\end{split}
			\end{align*}
			where for the second equality, we used \ref{itm:cl-sum-delta}. This establishes \ref{itm:cl-w}.
			
			For \ref{itm:cl-delta}, we first estimate $w(N(1,x))$ by calculating the weight of all $(r-2)$-tuples in $[t]\setminus \{1,x\}$ and then subtract the weight of the $(r-2)$-tuples $f$ such that $f \cup \{1,x\} \in \overline{G}$:
			\begin{align}\label{est}
				w(N(1,x)) 
				& = \binom{t-2}{r-2} \alpha^{r-2} + O\left(\sum_x \delta(x) \cdot \binom{t-3}{r-3} \alpha^{r-3}\right) + O\left(e(1,x) t^{-(r-2)}\right) \nonumber \\
				& = \mu_2 + O(at^{-(r-1)}),
			\end{align}
			where we use the definition of $\mu_2$ from \eqref{eqn:def-delta-1} and the fact that $e(1,x) \le e(1) \le ar/t$ (\Cref{prop:prelims-old} \ref{itm:useful-missing-1}).

			By \Cref{lem:easy} \ref{itm:frankl}, $\delta(x) - \delta(1) = w(1) - w(x) = \frac{w(N_1(x)) - w(N_x(1))}{w(N(1,x))}$. Since $G$ is left-compressed, it follows from \ref{itm:cl-w} that
			\[
				w(N_1(x)) - w(N_x(1)) 
				= w(N_1(x) \setminus N_x(1))
				= (e(x) - e(1)) \cdot t^{-(r-1)} \cdot \left(1 + O(a t^{-(r-2)})\right)
			\]
			Thus we obtain the required estimate for \ref{itm:cl-delta}:
			\begin{align*}
				\delta(x) 
				& = \frac{1}{w(N(1,x))} \cdot \frac{e(x) - e(1)}{t^{r-1}}\left(1 + O\left(a t^{-(r-2)}\right)\right) + \delta(1) \\
				& = \frac{e(x)}{\mu_2 t^{r-1}}\left(1 + O\left(a t^{-(r-2)}\right)\right). \qedhere
			\end{align*}
		\end{proof}
		
		To estimate $\lambda(G)$, we write $\lambda(G) = w(G) = w(\cl{t}{r}) - w(\overline{G}) = \left(w(\cl{t}{r}) - \lambda(\cl{t}{r}) \right) - w(\overline{G}) + \lambda(\cl{t}{r})$. We shall estimate each of these terms before putting everything together to complete the proof of \Cref{lem:eval-lag}, starting with $w(\overline{G})$.
		\begin{equation} \label{eqn:weight-comp}
		\begin{split}    
			w(\overline{G}) 
			=\, & \sum_{(x_1, \ldots, x_r) \in E(\overline{G})}w(x_1) \cdots w(x_r) \\
			=\, & \sum_{(x_1, \ldots, x_r) \in E(\overline{G})}\left( \alpha^r - \alpha^{r-1} \sum_{j \in [r]}\delta(x_j) + O\left(t^{-(r-2)} \left(a t^{-(r-1)}\right)^2\right) \right) \\
			=\, & \sum_{(x_1, \ldots, x_r) \in E(\overline{G})}\left( \frac{1}{t^r} + \frac{r^2 a}{\mu_2 t^{2r-1}} - \frac{1}{\mu_2t^{2(r-1)}}\sum_{j \in [r]}e(x_j) + O\left(a^2 t^{-3r + 4} \right)\right) \\
			=\, & \frac{a}{t^r} + \frac{r^2a^2}{\mu_2 t^{2r-1}} - \frac{1}{\mu_2 t^{2(r-1)}} \cdot \sum_x e(x)^2 + O\left( a^3 t^{-3r+4} \right).
		\end{split}
		\end{equation}
		Here we used \Cref{cl:eval-alpha-delta} \ref{itm:cl-sum-delta}-\ref{itm:cl-delta} and the following estimate, which uses \Cref{cl:eval-alpha-delta} \ref{itm:cl-alpha} and \ref{itm:cl-sum-delta}.
		\begin{align*}
			\alpha^s 
			=\, & \frac{1}{t^s}\left(1 + \sum_x \delta(x)\right)^s \\
			=\, & \frac{1}{t^s}\left(1 + s \cdot \sum_x \delta(x) + O\left(\left(a t^{-(r-1)}\right)^2\right)\right) \\
			=\, & \frac{1}{t^s}\left(1 + s \cdot \sum_x \frac{e(x)}{\mu_2 t^{r-1}}\right) + O\left(a^2 t^{-2r+3-s}\right) \\
			=\, & \frac{1}{t^s}\left(1 + \frac{s r a}{\mu_2 t^{r-1}}\right) + O\left(a^2 t^{-2r+3-s}\right).
		\end{align*}

		Next, we have the following estimate for $w(\cl{t}{r})$, using \Cref{cl:eval-alpha-delta} \ref{itm:cl-alpha}  and \ref{itm:cl-sum-delta}. 
		\begin{align*}
			w(\cl{t}{r})  
			& = \binom{t}{r}\alpha^r - \binom{t-1}{r-1}\alpha^{r-1} \sum_x \delta(x) + \binom{t-2}{r-2}\alpha^{r-2}\sum_{x<y}\delta(x) \delta(y) + O\left(\frac{a^3}{t^{3(r-1)}}\right).
		\end{align*}
		We now express $\lambda(\cl{t}{r})$ in a way that will make it easy to compare it with $w(G)$. Here we use the fact that the Lagrangian of the clique $\cl{t}{r}$ is the weight of the clique with respect to the uniform weighting, where every vertex has weight $\frac{1}{t} = \alpha - \frac{1}{t}\sum_x \delta(x)$. We have the following, again using \Cref{cl:eval-alpha-delta} \ref{itm:cl-sum-delta}.
		\begin{align*}
			\lambda(\cl{t}{r})
			& = \binom{t}{r}\left(\alpha - \frac{1}{t}\sum_x \delta(x)\right)^r \\
			& = \binom{t}{r}\alpha^r - \binom{t}{r}\cdot \frac{\alpha^{r-1}r}{t}\sum_x \delta(x) + \binom{t}{r}\alpha^{r-2}\binom{r}{2}\!\! \left(\frac{1}{t} \sum_x \delta(x)\right)^2\!\! + O\!\left(\frac{a^3}{t^{3(r-1)}}\right)\\
			& = \binom{t}{r}\alpha^r - \binom{t-1}{r-1}\alpha^{r-1}\! \sum_x\!\delta(x) + \frac{t-1}{2t}\binom{t-2}{r-2}\alpha^{r-2}\!\left(\!\sum_x\!\delta(x)\!\!\right)^2\!\! + O\!\left(\frac{a^3}{t^{3(r-1)}}\right).
		\end{align*}
		Now let us estimate the difference between the Lagrangian and the weight of $\cl{t}{r}$.
		\begin{align} \label{eqn:diff-weight-clique}
		\begin{split}
			& \lambda(\cl{t}{r}) - w(\cl{t}{r}) \\ 
			=\, & \binom{t-2}{r-2}\alpha^{r-2} \cdot \frac{1}{2}\left(\sum_x \delta(x)^2 - \frac{1}{t} \left(\sum_x \delta(x)\right)^2\right) + O\left(\left(a t^{-(r-1)}\right)^3\right) \\
			=\, & \binom{t-2}{r-2}\frac{1}{t^{r-2}} \cdot \frac{1}{2\mu_2^2 t^{2(r-1)}} \left(\sum_x e(x)^2 - \frac{1}{t}\left(\sum_x e(x)\right)^2\right) + O\left(a^3 t^{-3r+4}\right) \\
			=\, & \frac{1}{2\mu_2 t^{2(r-1)}} \left(\sum_x e(x)^2 - \frac{r^2a^2}{t}\right) + O\left(a^3 t^{-3r+4}\right),
		\end{split}
		\end{align}
		where we used \Cref{cl:eval-alpha-delta} \ref{itm:cl-alpha}, \ref{itm:cl-sum-delta} and \ref{itm:cl-delta}.

		By \eqref{eqn:weight-comp} and \eqref{eqn:diff-weight-clique}, we have that $\lambda(\cl{t}{r}) - \lambda(G)$ is
		\begin{align*}
			=\, & \lambda(\cl{t}{r}) - w(\cl{t}{r}) + w(\overline{G}) \\
			=\, & \frac{1}{2\mu_2 t^{2(r-1)}}\left( \sum_x e(x)^2 - \frac{r^2a^2}{t}\right)+ \left(\frac{a}{t^r} + \frac{r^2a^2}{\mu_2 t^{2r-1}} - \frac{1}{\mu_2 t^{2(r-1)}} \cdot \sum_x e(x)^2\right) + O\big(a^3 t^{-3r+4}\big) \\
			=\, & \frac{a}{t^r} - \frac{1}{2\mu_2 t^{2(r-1)}} \cdot \sum_x e(x)^2 + \frac{r^2a^2}{2\mu_2 t^{2r-1}} + O\big(a^3 t^{-3r+4}\big).
		\end{align*}
		\Cref{lem:eval-lag} follows, as $\lambda(\cl{t}{r}) = \mu_0$.
	\end{proof}

\section{Conclusion} \label{sec:conc}
	In this paper we proved that the Lagrangian of $3$-graphs on $m$ edges is maximised by an initial segment of colex, thus confirming the Frankl-F\"uredi conjecture for $r = 3$. For $r \ge 4$, we showed in \cite{Us1} that colex graphs are not always maximisers. Nevertheless, we obtain some information about the structure of a maximiser. In particular, we show that for most values of $m$ with $\binom{t}{r} - \binom{t-i}{r-i} \le m \le \binom{t}{r}$ with $2 \le i \le r-1$ any maximiser of the Lagrangian among $r$-graphs with $m$ edges is isomorphic to a subgraph of $\cl{t}{r}$ whose non-edges all contain a fixed set of size $i$. It seems plausible that this statement should hold for all such $m$. 
	We also show that for every $m$ with $\binom{t}{r} - o(t^{r-2}) \le m \le \binom{t}{r}$, every maximiser is a subgraph of $\cl{t}{r}$ whose complement is close to maximising the sum of degrees squared (among $r$-graphs with $\binom{t}{r}-m$ edges and $t$ vertices). It is plausible that for every maximiser of the Lagrangian, the complement maximises sum of degrees squared precisely, and not just asymptotically, and moreover that this holds for all $m$ with $\binom{t}{r} - \binom{t-2}{r-2} < m \le \binom{t}{r}$ (for other values of $m$ in $[\binom{t-1}{r}, \binom{t}{r}]$ we know that $\cl{t-1}{r}$ maximises the Lagrangian). We note that these problems are likely to be hard, and even if they were proven true, the task of characterising the maximisers of the Lagrangian for all $m$ and $r$ appears out of reach. As evidence, one could consider the question of determining the maximum sum of degrees squared, among $r$-graphs with given size and order, a question which is open for $r \ge 3$.

	\subsection*{Acknowledgements}
		This research was partially completed while the third author was visiting ETH Zurich. The third author would like to thank Benny Sudakov and the London Mathematical Society for making this visit possible. 
		We would like to thank Rob Morris and Imre Leader for their helpful comments and advice.

	\bibliography{lagra}
    \bibliographystyle{amsplain}

\appendix

\section{Proof of Nikiforov's Conjecture}\label{app:nik}
	\begin{thm}
		Let $\binom{t-1}{r} < m \le \binom{t}{r}$ and let $x$ be such that $\binom{x}{r} = m$ (so $x$ need not be an integer). For $t$ sufficiently large, we have $\Lambda(m, r) \le m x^{-r}$, with equality if and only if $x$ is an integer.
	\end{thm}

	\begin{proof}
		Let $G$ be a maximiser of the Lagrangian among $r$-graphs with $m$ edges. Our aim is to show that $\lambda(G) \le m x^{-r}$. By \Cref{thm:t-vs}, we may assume that the vertex set of $G$ is $[t]$. Write $m = \binom{t}{r} - a$. By \Cref{thm:main1}, if $a \ge \binom{t-2}{r-2}$, then 
		\begin{equation*}
			\lambda(G) \le \lambda(\cl{t-1}{r}) \le \binom{t-1}{r} \left( \frac{1}{t-1} \right)^r \le m x^{-r},
		\end{equation*}
		as $m \ge \binom{t-1}{r}$ and so $x \ge t-1$. We may thus assume that $a < \binom{t-2}{r-2}$. Again by \Cref{thm:main1}, we may assume that $a > 0$.

		Now, let us look more closely at the expression $m x^{-r}$. Write $x = t - \eps$, then clearly $0 \le \eps < 1$. Let $f$ be the function defined by $f(s) = \binom{s}{r}$ for $s \ge r$. Then
		\begin{align} \label{eqn:a-eps}
			a 
			= \binom{t}{r} - m 
			= \binom{t}{r} - \binom{x}{r} 
			= f(t) - f(x) 
			\le (t - x) f'(t) 
			\le \eps \binom{t}{r-1}.
		\end{align}
		Here we used the fact that $f'(s)$ is increasing and that $f'(t) \le \binom{t}{r-1}$. We now obtain a lower bound on $m x^{-r}$. 
		\begin{align*}
			m x^{-r} 
			& = m t^{-r} \left( \frac{1}{1 - \eps/t} \right)^r \\
			& \ge m t^{-r} ( 1 + r \eps/t) \\
			& \ge m t^{-r} \left( 1 + \frac{ra}{\binom{t}{r-1} t} \right) \\
			& = \binom{t}{r}t^{-r} - a t^{-r} \left( 1 - \frac{\binom{t}{r} r}{\binom{t}{r-1}t} + \frac{ra}{\binom{t}{r-1}t} \right) \\
			& \ge \lambda(\cl{t}{r}) - a \cdot t^{-(r+1)} (r-1) (1 - r/t),
		\end{align*}
		where we used \eqref{eqn:a-eps} and the upper bound $a \le \binom{t-2}{r-2}$.
		In contrast, recall that by \Cref{prop:prelims-new} \ref{itm:useful-w-t-large} $w(t) = \Omega(t^{-1})$, hence the weight of every non-edge of $G$ is at least $\Omega(t^{-r})$. We thus have
		\begin{align*}
			w(G) = w(\cl{t}{r}) - w(\overline{G}) \le \lambda(\cl{t}{r}) - \Omega(a t^{-r}) \le m x^{-r},
		\end{align*}
		as required. Note that we get equality if and only if $a = 0$, that is, when $x = t$.
	\end{proof}

\end{document}